\documentclass{article}
\usepackage{amsfonts}
\usepackage{amsmath}
\usepackage{amsthm}
\usepackage{enumerate}
\usepackage{indentfirst}
\usepackage{color}

\newtheorem{lem}{Lemma}[section]

\newtheorem{theorem}{Theorem}[section]

\newtheorem{coro}{Corollary}[section]
\newcommand{\ppp}{\partial}

\newcommand{\ooo}{\overline}
\newcommand{\OOO}{\Omega}

\renewenvironment{abstract}{%
        \small
        \quotation
         \noindent {\bfseries \abstractname } }%
      {\if@twocolumn\else\endquotation\fi}

\title{\bf Unique continuation principle for the one-dimensional 
time-fractional diffusion equation
}
\author{
Zhiyuan Li$^\dag$ ,\quad Masahiro Yamamoto$^\ddag$ 
}
\date{}

\begin{document}
\maketitle
\renewcommand{\thefootnote}{\fnsymbol{footnote}}
\footnotetext{\hspace*{-5mm} 
\begin{tabular}{@{}r@{}p{14cm}@{}} 
& Manuscript last updated: \today.
\\
$^\dag$ 
& School of Mathematics and Statistics, Shandong University of Technology, 
Zibo, Shandong 255049, China.
E-mail: zyli@sdut.edu.cn\\
$^\ddag$ 
& Graduate School of Mathematical Sciences, 
the University of Tokyo,
3-8-1 Komaba, Meguro-ku, Tokyo 153-8914, Japan. 
E-mail: myama@ms.u-tokyo.ac.jp.
\end{tabular}}

\begin{abstract}
This paper deals with the unique continuation of solutions for 
a one-dimensional anomalous diffusion equation with Caputo derivative of order 
$\alpha\in(0,1)$. Firstly, the
uniqueness of solutions to a lateral Cauchy problem for the anomalous 
diffusion equation is given via the Theta function method, from which we 
further verify the unique continuation principle.
\end{abstract}

\section{Introduction and main result}
\label{sec-intro-distri}

The anomalous diffusion processes whose mean square displacement behaves like 
$\langle\Delta x^2\rangle\sim C_\alpha t^\alpha$ as $t\to\infty$ were found in 
many problems in the fields of science and engineering. For the qualitative 
analysis of these anomalous diffusion, a macro-model based on the continuous 
time random walk, which is called a time-fractional diffusion equation,
is derived:
\begin{equation}
\label{equ-gov}
\partial_t^\alpha u(x,t) - \partial_x^2u(x,t) = 0,\quad (x,t)\in(0,1)\times
(0,T),
\end{equation}
with the Caputo derivative $\partial_t^\alpha$ $(0<\alpha<1)$ which is 
usually defined by
$$
\partial_t^\alpha \phi(t) 
= \frac{1}{\Gamma(1-\alpha)} \int_0^t \frac{\phi'(\tau)}{(t-\tau)^\alpha}d\tau,
\quad t>0,
$$
where $\Gamma(\cdot)$ is a usual Gamma function. For various properties of the 
Caputo derivative, we refer to Kilbas, Srivastava and Trujillo \cite{K06}, 
Podlubny \cite{P99} and the references therein.

The fractional diffusion models have received great attention  
in applied disciplines, e.g., in describing some anomalous phenomena including 
the non-Fickian growth rates, skewness and long-tailed profile 
which are poorly characterized by the classical diffusion 
equations (see e.g., Benson, Wheatcraft and Meerschaert \cite{B00}, 
Levy and Berkowitz \cite{LB03} and the references therein). 
In contrast to the success in the practice, theoretical researches related to 
the fractional diffusion equation are still under development. The Caputo 
derivative is inherently nonlocal in time with a history dependence, and
there are crucial differences between fractional models and classical models 
(i.e., $\alpha=1$), for example, concerning long-time asymptotic behavior 
(see, e.g., Li, Luchko and Yamamoto \cite{LLY14} and Li, Liu and Yamamoto 
\cite{LLY15}). There are also 
some publications on some important properties. 
For example, a maximum principle in the usual setting still holds 
similarly to the parabolic equation (see, e.g., Luchko \cite{L09}). 
As is known, the unique continuation property (UCP) is one of remarkable 
properties of parabolic equations, which asserts that if a solution 
to a homogeneous equation vanishes in an open subset, then the solution is identically 
zero in the whole domain (see, e.g., Saut and Scheurer \cite{SS87}). 
For the time-fractional diffusion equation, there are not affirmative answers 
except for some special cases. For the special half-order fractional diffusion 
equation (i.e., $\alpha=\frac12$), under the assumption that the initial-value 
vanishes, in Xu, Cheng and Yamamoto \cite{XCY11} for the one-dimensional case,
and Cheng, Lin and Nakamura \cite{CLN13} for the two-dimensional case, 
the uniqueness results are
proved: if a solution $u$ to a homogeneous fractional diffusion equation 
satisfies $u=0$ in $\omega\times (0,T)$ and $u(\cdot,0) =0$ in $\OOO$
with some subdomain $\omega \subset \OOO$, then $u=0$ in $\OOO\times (0,T)$.
Here $\OOO \subset \mathbb R^d$ is a spatial domain where we are considering 
the fractional diffusion equation.
The proof is done via Carleman estimates for the operator 
$\partial_t-\Delta^2$. For a general fractional order $\alpha \in (0,1)$, 
a recent work Lin and Nakamura \cite{LN16} obtained a similar 
uniqueness result with the zero initial condition by using a newly established 
Carleman estimate based on calculus of pseudo-differential operators. Sakamoto 
and Yamamoto \cite{SY11} showed that for a solution of the time-fractional 
diffusion equations with the homogeneous Dirichlet boundary condition on 
the whole boundary, if the Neumann data vanish on arbitrary subboundary, 
then it vanishes identically.  The paper Jiang, Li, Liu and Yamamoto 
\cite{JLLY} generalized the result in \cite{SY11} to 
the multi-term case. Recently, for the multi-term case with the first order 
time-derivative, the UCP was established by Li, Huang and Yamamoto 
\cite{LHY17} 
via a Carleman type estimate for the parabolic equation. All these results 
should be considered as a weak type of uniqueness because the homogeneous 
boundary condition is imposed on the whole 
boundary (\cite{JLLY} and \cite{SY11}) or 
on the initial value (\cite{CLN13}, \cite{LHY17}, \cite{LN16} and 
\cite{XCY11}).

In this paper, we will show that the classical unique continuation property 
for solutions of \eqref{equ-gov} is valid. More precisely, we have the 
following main 
theorem.
\begin{theorem}
\label{theorem-unique}
Let $T>0$ be fixed constant and $u\in L^\infty(0,T;H^2(0,1))$ be a solution to 
the fractional diffusion equation \eqref{equ-gov}. Then we have
$$
u(x,t)=0,\quad (x,t)\in[0,1]\times[0,T]
$$
provided that $u\equiv0$ in $I\times[0,T]$, where $I$ is a non-empty open 
subinterval of $(0,1)$. 
\end{theorem}

In the theorem, we consider a class of solutions satisfying
$u \in L^{\infty}(0,T;H^2(0,1))$ and $\ppp_t^{\alpha}u 
\in L^{\infty}(0,T;L^2(0,1))$.  In the case where $u(0,t) = u(1,t) = 0$
for $t>0$ and $u(\cdot,0) \in H^1_0(0,1) \cap H^2(0,1)$, we can prove
that $u \in L^{\infty}(0,T;H^2(0,1))$ (e.g., \cite{SY11}).

This theorem is exactly corresponding to the unique continuation in the
case of $\alpha=1$, not assuming the zero initial condition like  
\cite{CLN13}, \cite{LN16}, \cite{XCY11}, but we do not know if the same 
unique continuation holds for general dimensions or equations with 
variable coefficients even in the one dimension.

\section{Proof of Theorem \ref{theorem-unique}}
In this section, we will set up notations and terminologies, review some of 
standard facts on the fractional calculus, and introduce a lateral Cauchy 
problem for the time-fractional diffusion equation. Preparing all the 
necessary properties of the solution of the Cauchy problem, we finally give 
the proof of the main result.

Besides the Caputo derivatives, we also use the Riemann-Liouville fractional 
derivatives which is defined as follows
$$
D_t^\alpha \phi(t)
:=\frac1{\Gamma(1-\alpha)} 
\frac{d}{dt} \int_0^t \frac{\phi(\tau)}{(t-\tau)^\alpha} d\tau,\quad 
t>0,\ \alpha\in(0,1). 
$$

Let $\delta(x)$ be the Dirac delta fucntion and let $K_\alpha(x,t)$ be 
the fundamental solution of the following free space 
time-fractional diffusion equation
\begin{equation}
\label{equ-free}
\left\{
\begin{alignedat}{2}
&\partial_t^\alpha u - \partial_x^2u = 0,
&\quad& x\in\mathbb R,\ t>0,
\\
&u(x,0) = \delta(x), &\quad& x\in\mathbb R,
\end{alignedat}
\right.
\end{equation}
which has the form from Luchko and Zuo \cite{LZ14} and Rundell, Xu and Zuo 
\cite{R13}
$$
K_\alpha(x,t) 
= \tfrac12 t^{-\frac{\alpha}2} M_{\frac{\alpha}2}(|x|t^{-\frac{\alpha}{2}}),
\quad x\in\mathbb R,\ t>0,
$$
where $M_\alpha$ is a special member of the family of the Wright functions 
(see, e.g., Mainardi, Luchko and Pagnini \cite{MLP01} and \cite{LZ14}) 
which is defined by
$$
M_\alpha(t)=\sum_{k=0}^\infty \frac{(-t)^k}{k!\Gamma(-\alpha k + (1-\alpha))}
$$
and its Laplace transform with respect to the time $t$ has the form
\begin{equation}
\label{equ-Lap_K}
\mathcal{L}\{K_\alpha(x,t);s\}
=\tfrac12 s^{\frac{\alpha}2-1} e^{-|x|s^{\frac{\alpha}2}},
\quad s>0,\ x\in\mathbb R.
\end{equation}
Moreover, based on the fundamental solution $K_\alpha$, we introduce an 
important special function named Theta function $\theta_\alpha(x,t)$, 
$\alpha>0$ which plays an important role in representing the solution to 
\eqref{equ-gov} with non-homogeneous boundary conditions. 
We consider the Theta function $\theta_\alpha$ 
by the following form of series:
$$
\theta_\alpha(x,t) := \sum_{m=-\infty}^\infty K_\alpha(x+2m,t),\quad t>0.
$$
We now list some properties of the function $\theta_\alpha$ which can be found 
in Eidelman and Kochubei \cite{EK03}, Luchko and Zuo \cite{LZ14} and Rundell, 
Xu and Zuo \cite{R13}.
\begin{lem}
\label{lem-theta}
The functions $K_\alpha(x,t)$ and $\theta_{\alpha}(x,t)$ are 
even with respect to $x$, and $\theta_\alpha(1,t) = \theta_\alpha(-1,t)$ 
is $C^\infty$ on $[0,\infty)$ and 
$$
\frac{d^m\theta_\alpha(1,0)}{dt^m}=0,\quad m=0,1,\cdots.
$$
Moreover, $\theta_\alpha$ satisfies the following estimates.
\begin{itemize}
\item[$(a)$] If $|x|^2\ge t^\alpha>0$, then there exist constants $C>0$ and 
$\sigma>0$ depending on $\alpha$ such that
\begin{equation}
\label{esti-theta}
|K_\alpha(x,t)| \le 
Ct^{-\frac{\alpha}2} e^{-\sigma t^{-\frac{\alpha}{2-\alpha}}|x|^{\frac2{2-\alpha}}},
\end{equation}
and
\begin{equation}
\label{esti-theta'}
|D_t^{1-\alpha} K_\alpha (x,t)| \le 
Ct^{\frac{\alpha}2-1}e^{-\sigma t^{-\frac{\alpha}{2-\alpha}}|x|^{\frac2{2-\alpha}}}.
\end{equation}
\item[$(b)$]
If $0<|x|^2\le t^\alpha$, then there exists a constant $C>0$ depending on $\alpha$ 
such that
\begin{equation}
\label{esti-theta<1}
|K_\alpha(x,t)| \le C t^{-\frac{\alpha}2},
\end{equation}
and
\begin{equation}
\label{esti-theta'<1}
|D_t^{1-\alpha} K_\alpha(x,t)| \le Ct^{\frac{\alpha}2 - 1 }.
\end{equation}
\end{itemize}
\end{lem}

Now let us turn to considering the following lateral Cauchy problem
\begin{equation}
\label{equ-cauchy}
\left\{
\begin{alignedat}{2}
&\partial_t^\alpha u - \partial_x^2u = 0
&\quad& \mbox{in $(0,1)\times(0,T]$,}
\\
&u(0,\cdot)= u_x(0,\cdot) = 0 &\quad& \mbox{in $[0,T]$}.
\end{alignedat}
\right.
\end{equation}
Assuming that $u\in L^\infty(0,T;H^2(0,1))$ satisfies (\ref{equ-cauchy}),
we will now focus on the representation of the solution to \eqref{equ-cauchy}, 
as this 
will be essential to our approach. For this, we first set $u_0(x):=u(x,0)$ and 
$g(t):=u_x(1,t)$. We extend the function $g$ to the interval $[0,\infty)$ by 
letting $g=0$ outside of $(0,T+1)$ and letting 
$g(t) = g(T)(T+1-t)$ if $t\in(T,T+1)$, 
and by $\widetilde g$ we denote the extension, and by $\widetilde u$ we denote
the solution to the following auxiliary system
\begin{equation}
\label{equ-IBVP}
\left\{
\begin{alignedat}{2}
&\partial_t^\alpha \widetilde u - \partial_x^2 \widetilde u = 0
&\quad& \mbox{in $(0,1)\times(0,\infty)$,}
\\
&\widetilde u(\cdot,0)=u_0, &\quad& \mbox{in $(0,1)$,}
\\
&\widetilde u_x(0,\cdot) = 0,\  
\widetilde u_x(1,\cdot)=\widetilde g&\quad& \mbox{in $(0,\infty)$.}
\end{alignedat}
\right.
\end{equation}
On the basis of the properties of the fundamental solution $K_\alpha$ and the 
Theta function $\theta_\alpha$ in Lemma \ref{lem-theta}, a representation 
formula of the solution to \eqref{equ-IBVP} can be obtained. We have
\begin{lem}
\label{lem-u}
Assume $u_0\in C[0,1]$ and $\widetilde g$ be a continuous function.
Then the solution $\widetilde u$ of the initial-boundary value problem 
\eqref{equ-IBVP} has the form
\begin{equation}
\label{equ-u}
\widetilde u(x,t)=w(x,t)+v(x,t),\quad (x,t)\in(0,1)\times(0,\infty),
\end{equation}
where 
\begin{align}
w(x,t)&=\int_0^1 (\theta_\alpha(x-\xi,t) + \theta_\alpha(x+\xi,t)) 
u_0(\xi)d\xi,
\label{sol-w}
\\
v(x,t)&=2\int_0^t(D_t^{1-\alpha}\theta_\alpha)(x-1,t-\tau)\widetilde g(\tau)
d\tau,
\quad (x,t)\in(0,1)\times(0,\infty). 
\label{sol-v1} 
\end{align}
Moreover, the following estimate 
$$
|\widetilde u(0,t)|\le 
C\left( t^{-\frac{\alpha}2} + t^{\frac{\alpha}2} + t^{\frac{3\alpha}2 } 
+ t^{\frac{\alpha^2}{2-\alpha} }  + t^{\frac{2\alpha-2}{2-\alpha} }\right),
\quad t>0.
$$
holds true for $t>0$.
\end{lem}
\begin{proof}
The representation formula \eqref{equ-u} is directly derived from 
Lemma 3.1 in \cite{R13}. In order to finishing the proof of the lemma, 
it remains to show the estimates for $w$ and $v$. For this, from Lemma 
\ref{lem-theta}, we see that the Theta function $\theta_\alpha$ is even with 
respect to $x$. Thus
$$
w(0,t) =2 \int_0^1 \theta_\alpha(\xi,t) u_0(\xi)d\xi.
$$ 
We need to evaluate $\theta_\alpha(x,t)$ for $(x,t)\in (0,1)\times(0,\infty)$. 
In the case of $|x+2m|^2\ge t^\alpha$, from the estimate \eqref{esti-theta},
it follows that 
$$
|K_\alpha(x+2m,t)| \le 
Ct^{-\frac{\alpha}2}e^{-\sigma t^{-\frac{\alpha}{2-\alpha}}
|x+2m|^{\frac2{2-\alpha}} }.
$$
Moreover, if $m=0$, then we have
$$
|K_\alpha(x,t)| \le
\left\{
\begin{alignedat}{2}
&Ct^{-\frac{\alpha}2} e^{-\sigma t^{-\frac{\alpha}{2-\alpha}} x^{\frac2{2-\alpha}} }
\le Ct^{-\frac{\alpha}2} e^{-\sigma}
\le Ct^{-\frac{\alpha}2},&\quad& x^2\ge t^\alpha,
\\
&Ct^{-\frac{\alpha}2},&\quad& x^2\le t^\alpha.
\end{alignedat}
\right.
$$
For $m\ne0$, we have $|x+2m|\ge 2|m| - 1\ge 1$ by $0 < x < 1$, 
which further implies that 
$|x+2m|^{\frac2{2-\alpha}} \ge 2|m| - 1$,
because of $\frac2{2-\alpha} >1$. Consequently
$$
|K_\alpha(x+2m,t)| 
\le Ct^{-\frac{\alpha}2} e^{-\sigma t^{-\frac{\alpha}{2-\alpha}} (2|m|-1) },
\quad |x+2m|^2\ge t^\alpha,\ m\ne0.
$$
On the other hand, if $|x+2m|^2 < t^\alpha$, then the estimate 
\eqref{esti-theta<1} implies
$$
|K_\alpha(x+2m,t)| 
\le Ct^{-\frac{\alpha}2},\quad |x+2m|^2< t^\alpha.
$$
Collecting all the above estimates, we obtain 
\begin{align*}
|\theta_\alpha(x,t)|
\le& |K_\alpha(x,t)|
+ \left(\sum_{|x+2m|^2 \ge t^\alpha,m\ne0} 
+ \sum_{|x+2m|^2 < t^\alpha,m\ne0} \right)\vert K_\alpha(x+2m,t)\vert
\\
\le& Ct^{-\frac{\alpha}2} + \sum_{|x+2m|^2 \ge t^\alpha,m\ne0} 
Ct^{-\frac{\alpha}2} e^{-\sigma t^{-\frac{\alpha}{2-\alpha}} (2|m|-1) }
+ \sum_{|x+2m|^2 < t^\alpha} Ct^{-\frac{\alpha}2},
\end{align*}
where $(x,t)\in(0,1)\times(0,\infty)$. Again by noting that $|x+2m|
\ge 2|m| -1$, we see that $|x+2m|^2 < t^\alpha$ implies 
$2|m| \le 1+ t^{\frac{\alpha}2}$, so that
$$
\sum_{|x+2m|^2 < t^\alpha} Ct^{-\frac{\alpha}2}
= Ct^{-\frac{\alpha}{2}}\sum_{\vert x+2m\vert^2 < t^{\alpha}}
\le Ct^{-\frac{\alpha}2} (2+ t^{\frac{\alpha}{2}}) 
= C\left( t^{-\frac{\alpha}2} + t^{\frac{\alpha}2}\right).
$$
By direct calculations, we find
$$
\sum_{|x+2m|^2\ge t^\alpha,m\ne0} e^{-\sigma t^{-\frac{\alpha}{2-\alpha}} 
(2|m|-1)} 
\le 2\sum_{m=1}^{\infty}  e^{-\sigma t^{-\frac{\alpha}{2-\alpha}} (2m-1) } 
= \frac{2e^{\sigma t^{-\frac{\alpha}{2-\alpha}}}}
{e^{2\sigma t^{-\frac{\alpha}{2-\alpha}}}-1}.
$$
Moreover we can directly verify that
$$
\frac{2e^{\sigma t^{-\frac{\alpha}{2-\alpha}}}}
{e^{2\sigma t^{-\frac{\alpha}{2-\alpha}}} -1}
\le C t^{\frac{\alpha}{2-\alpha}},\quad t>0,
$$
which implies that
$$
\sum_{|x+2m|^2 \ge 
t^\alpha,m\ne0} Ct^{-\frac{\alpha}2} e^{-\sigma t^{-\frac{\alpha}{2-\alpha}} 
(2|m|-1) } 
\le Ct^{\frac{\alpha^2}{2-\alpha}},\quad t>0.
$$
Finally we obtain
\begin{align*}
|\theta_\alpha(x,t)|\le 
C\left(t^{-\frac{\alpha}2}+t^{\frac{\alpha}2}+t^{\frac{\alpha^2}{2-\alpha}}\right),
\quad t>0.
\end{align*}
Similarly
\begin{equation}\label{13}
|D_t^{1-\alpha} \theta_\alpha(x,t)|\le 
C\left( t^{\frac{\alpha}2-1} + t^{\frac{3\alpha}2-1} 
+t^{\frac{3\alpha-2}{2-\alpha} } \right),\quad t>0.
\end{equation}
Therefore
$$
|w(0,t)|\le 
2C\left( t^{-\frac{\alpha}2} + t^{\frac{\alpha}2} 
+ t^{\frac{\alpha^2}{2-\alpha} } \right) \int_0^1 |u_0(\xi)| d\xi
\le 2C\|u_0\|_{L^1(0,1)}\left(t^{-\frac{\alpha}2} + t^{\frac{\alpha}2} 
+ t^{\frac{\alpha^2}{2-\alpha} } \right).
$$ 
In view of the definition of $v$, by direct calculations and (\ref{13}), 
we arrive at the estimate for $v(0,t)$: 
$$
|v(0,t)| \le 
2\|\widetilde g\|_{L^\infty(0,\infty)}\int_0^t|D_t^{1-\alpha}
\theta_\alpha(-1,\tau)|d\tau 
$$
\begin{equation}\label{esti-v(0)}
\le 2\|g\|_{L^\infty(0,T)}\left( \frac2{\alpha} t^{\frac{\alpha}2 } 
+ \frac2{3\alpha} t^{\frac{3\alpha}2 } 
+ \frac{2-\alpha}{2\alpha}t^{\frac{2\alpha-2}{2-\alpha} } \right),\quad t>0.
\end{equation}
Collecting all the above estimates, we finally find that
$$
|\widetilde u(0,t)|\le 
C\left( t^{-\frac{\alpha}2} + t^{\frac{\alpha}2} + t^{\frac{3\alpha}2 } 
+ t^{\frac{\alpha^2}{2-\alpha} }  + t^{\frac{2\alpha-2}{2-\alpha} }\right),
\quad t>0.
$$
\end{proof}
We also need a classical result from the complex analysis:
\begin{lem}[Phragm\'en-Lindel\"of principle]
\label{lem-PL}
Let $F(z)$ be a holomorphic function in a sector 
$S = \{ z\in\mathbb C; \theta_1 < \mbox{arg}\, z < \theta_2 \}$ of angle 
$\pi/\beta=\theta_1-\theta_2$, and continuous on the closure $\ooo{S}$.
If
\begin{equation}
\label{est-bdd}
|F(z)|\leq 1
\end{equation}
for $z \in \ppp S$: the boundary of $S$, and 
$$
|F(z)|\leq Ce^{C|z|^\gamma}
$$
for all $z \in S$, where $0\le\gamma<\beta$ and $C>0$, then \eqref{est-bdd} 
holds 
also for all $z$ in $S$.
\end{lem}
The proof of the above lemma can be found in Stein and Shakarchi \cite{SS03}. 
Furthermore, Phragm\'en-Lindel\"of principle yields the following useful 
corollary: 
\begin{coro}
\label{coro}
Let $f$ be a real-valued continuous function on the interval $[0,1]$ and 
satisfy the following estimate 
$$
\left|\int_0^1 f(t) e^{st}dt\right|\le Ce^{as},\quad s\ge0,
$$
where $0<a<1$ and the constant $C$ is independent of $s$. Then $f$ is 
identically zero on $[a,1]$. 
\end{coro}
\begin{proof}
By splitting $0\le t\le 1$ into the two parts $[0,a]$ and $[a, 1]$, 
we see that 
$$
\left|\int_a^1 f(t) e^{st}dt\right| = 
\left|\int_0^1 f(t) e^{st}dt - \int_0^a f(t) e^{st}dt\right|
\le (C + a\|f\|_{C[0,1]}) e^{as},\quad s\ge0. 
$$
After the change $\eta = t-a$ of variables, we arrive at 
$$
\left|\int_0^{1-a} f(t+a) e^{st}e^{as}dt\right| 
\le (C + a\|f\|_{C[0,1]}) e^{as},\quad s\ge0,
$$
hence that
$$
\left|\int_0^{1-a} f(t+a) e^{st}dt\right| \le C_1 ,\quad s\ge0,
$$
where $C_1:=C + a\|f\|_{C[0,1]}$.  Therefore our statement in this corollary 
is equivalent to the following: 
$$
\mbox{
If $G(z)=\int_0^b g(t)e^{tz}dt$ is bounded for $z=s\ge 0$, then $g\equiv 0$ in 
$[0,b]$.}
$$
Here $b:=1-a$ and $g(t):=f(t+a)$.

From the definition of the function $G$ and the above estimation, 
we see that $G$ is 
bounded on the imaginary axis and as well on $\{z\in\mathbb C; \arg z= 0\}$. 
Setting $\theta_1=0$ and $\theta_2=\frac{\pi}{2}$ in Lemma \ref{lem-PL} so that 
$\beta=2$, and noting the estimate 
$$
|G(z)|\le \int_0^b |g(t)| e^{t|z|}dt
\le b\|g\|_{L^\infty(0,b)} e^{b|z|},\quad 0<\arg z<\frac{\pi}2,
$$
we conclude from Lemma \ref{lem-PL} that $G$ must be bounded on the whole 
sector $\{z\in\mathbb C; 0<\arg z<\frac{\pi}{2}\}$. 
Similarly, $G$ must be bounded on the 
sector $\{z\in\mathbb C; -\frac{\pi}{2}<\arg z<0\}$. Thus $G$ is 
bounded on the right half plane.  Moreover we can directly see that 
$G(z)$ is bounded on the left half plane.   Thus, since $G$ is holomorphic 
on $\mathbb C$, iy follows that $G$ is a constant function, 
and the constant is zero, because 
$$
\lim_{z\to-\infty} \int_0^b g(t)e^{tz}\, dt = 0, 
$$
which implies our desired conclusion: $f\equiv0$ in $[a,1]$.
\end{proof}

Now we are ready to prove the uniqueness of solutions to the lateral 
Cauchy problem \eqref{equ-cauchy}. We have
\begin{lem}
\label{lem-Cauchy}
Let $T>0$ be a fixed constant and $u\in L^\infty(0,T;H^2(0,1))$ be a solution 
to the lateral Cauchy problem \eqref{equ-cauchy}. Then we have
$$
u(x,t)=0,\quad (x,t)\in[0,1]\times[0,T].
$$
\end{lem}

Theorem 1.1 directly follows from Lemma 2.4.  Indeed, setting $I=(a,b)$ with 
$0 < a < b < 1$, by $u\vert_{I\times (0,T)} = 0$, we have
$u(a,\cdot) = u_x(a,\cdot) = 0$ and $u(b,\cdot) = u_x(b,\cdot) = 0$ in 
$(0,T)$.  Changing independent variables $x \to a-x$ and 
$x \to x-b$ in the intervals $(0,a)$ and $(b,1)$ respectively,
and applying Lemma 2.4, we obtain $u=0$ in $(0,a) \times (0,T)$ and 
$(b,1) \times (0,T)$.

Thus the rest of the paper is devoted to the proof of Lemma 2.4.

\begin{proof}
From the above calculation and settings, and noting Lemma \ref{lem-u}, we see 
that 
$\widetilde u$ is an extension of $u$ which solves the Cauchy problem 
\eqref{equ-cauchy}, that is, $\widetilde u=u$ in $[0,1]\times[0,T]$. Using the 
assumption that $u(0,t)=0$ for $t\in[0,T)$, we find
\begin{align}
\label{equ-h}
2\int_0^1 \theta_\alpha(\xi,t) u_0(\xi)d\xi 
+ 2\int_0^t D_t^{1-\alpha}\theta_\alpha(1,t-\tau)\widetilde g(\tau)d\tau
= \left\{
\begin{alignedat}{2}
&0, &\quad &t\in(0,T),\\
&\widetilde u(0,t), &\quad& t\in[T,\infty).
\end{alignedat}
\right.
\end{align}
Taking the Laplace transforms on both sides of the above equation, we have
\begin{align}
\label{equ-Lap_h}
2\int_0^1 \mathcal L\{\theta_\alpha(\xi,t);s\} u_0(\xi)d\xi 
+ 2\mathcal L\{ D_t^{1-\alpha}\theta_\alpha(1,t);s\} 
\mathcal L\{\widetilde g(t);s\}
=\int_T^\infty \widetilde u(0,t)e^{-st}dt.
\end{align}
We will show several useful estimates which mainly describe the rate of the 
convergence of the terms in \eqref{equ-Lap_h} as $s\to\infty$.

First, from the definition of the Theta function and the formula 
\eqref{equ-Lap_K}, it follows that
\begin{align*}
2\mathcal L\{\theta_\alpha(x,t); s\} &=
s^{\frac{\alpha}2-1}\sum_{m=-\infty}^\infty e^{-|x+2m|s^{\frac{\alpha}2}},
\quad s>0.
\end{align*}
For $x\in[0,1]$, we further treat the above identity as follows
\begin{align*}
2\mathcal L\{\theta_\alpha(x,t); s\} 
&=s^{\frac{\alpha}2-1} 
\left( e^{xs^{\frac{\alpha}2}} \sum_{m=-\infty}^{-1} e^{2m s^{\frac{\alpha}2}} 
+ e^{-xs^{\frac{\alpha}2}} \sum_{m=0}^\infty e^{-2m s^{\frac{\alpha}2}} \right)
\\
&=s^{\frac{\alpha}2-1} 
\left(e^{xs^{\frac{\alpha}2}} \sum_{m=1}^\infty e^{-2m s^{\frac{\alpha}2}}
+ e^{-xs^{\frac{\alpha}2}} \sum_{m=0}^\infty e^{-2m s^{\frac{\alpha}2}} \right)
\\
&=s^{\frac{\alpha}2-1}\frac{e^{xs^{\frac{\alpha}2}}}{
e^{2 s^{\frac{\alpha}2}}-1}
+s^{\frac{\alpha}2-1}\frac{e^{(2-x)s^{\frac{\alpha}2}}}
{e^{2 s^{\frac{\alpha}2}}-1},
\quad s>0.
\end{align*}
From Lemma \ref{lem-theta}, we have $\theta_\alpha(1,0)=0$. 
Then using the formula
$$
\mathcal L \{ D_t^{1-\alpha}\theta_\alpha(1,t); s \}
=s^{1-\alpha} \mathcal L \{ \theta_\alpha(1,t); s \},
$$
we are led to 
$$
\mathcal L\{D_t^{1-\alpha}\theta_\alpha(1,t); s\} 
= s^{-\frac{\alpha}2} \frac{e^{s^{\frac{\alpha}2}}}
{e^{2 s^{\frac{\alpha}2}}-1 }, 
\quad s>0.
$$
Solving \eqref{equ-Lap_h} with respect to $\mathcal L\{\widetilde{g}(t);s\}$
and substituting the above representations of 
$\mathcal L\{\theta_{\alpha}(\xi,t);s\}$ and 
$\mathcal L\{D_t^{1-\alpha}\theta_{\alpha}(1,t);s\}$,
we have
\begin{align*}
\mathcal L\{\widetilde g(t);s\} 
=&\frac12 s^{\frac{\alpha}2} \left(e^{s^{\frac{\alpha}2}} - 
e^{-s^{\frac{\alpha}{2}}} \right) \int_T^\infty \widetilde u(0,t) e^{-st}dt 
- \frac12 s^{\alpha-1} \int_0^1 e^{(\xi-1) s^{\frac{\alpha}2}} u_0(\xi)d\xi 
\\
&- \frac12 s^{\alpha-1} \int_0^1 e^{(1-\xi) s^{\frac{\alpha}2}} u_0(\xi)d\xi \\
=: &I_1(s) + I_2(s) - I_3(s),\quad s>0,
\end{align*}
that is,
\begin{align*}
& I_3(t) = \frac{1}{2}s^{\alpha-1}\int^1_0 e^{(1-\xi)s^{\frac{\alpha}{2}}}
u_0(\xi) d\xi\\
=& \frac{1}{2}s^{\frac{\alpha}{2}}(e^{s^{\frac{\alpha}{2}}}
- e^{-s^{\frac{\alpha}{2}}})\int^{\infty}_{T} \widetilde{u}(0,t)
e^{-\xi t} dt\\
- & \frac{1}{2}s^{\alpha-1}\int^1_0 e^{(\xi-1)s^{\frac{\alpha}{2}}}
u_0(\xi) d\xi - \mathcal L\{ \widetilde{g}(t);s\}
\end{align*}
\begin{equation}\label{equ-sLg}
=: I_1(s) + I_2(s) - \mathcal L\{ \widetilde{g}(t);s\}.
\end{equation}

From the choice of the extension $\widetilde g$ of the function $g$, 
we conclude that the left-hand side of the above equation can be rephrased 
as follows
$$
\mathcal L\{\widetilde g(t);s\} 
= \int_0^T g(t) e^{-st}dt + g(T)\int_T^{T+1} (T+1-t) e^{-st}dt,
$$
which implies that
$$
|\mathcal L\{\widetilde g(t);s\} |
\le \|g\|_{L^\infty(0,T)} \int_0^{T+1}  e^{-st}dt
= \frac{\|g\|_{L^\infty(0,T)} (1-e^{-s(T+1)}) }{s}
\le \|g\|_{L^\infty(0,T)} s^{-1},\quad s>0.
$$
Here we used that $g = u_x(1,\cdot) \in L^{\infty}(0,T)$ 
by $u_x\in L^{\infty}(0,T;H^1(0,1))$ and $H^1(0,1) \subset 
C[0,1]$.  Now by letting $s$ sufficiently large, we conclude that
\begin{equation}
\label{lim-sLg}
\mathcal L\{\widetilde g(t);s\} \to 0,\quad \mbox{as $s\to\infty$.}
\end{equation}
The final conclusion of Lemma \ref{lem-u} yields 
$|\widetilde u(0,t)|<Ce^{Mt}$ for $t\in [T,\infty)$, where $C,M>0$ are 
constants only depend on $\alpha$, $T$, $u_0$ and $g$, which implies that
$$
\left|\int_T^\infty \widetilde u(0,t) e^{-st}dt\right|
\le \int_T^\infty Ce^{(M-s)t}dt
=\frac{Ce^{MT}}{s-M} e^{-sT},\quad s>2M.
$$
By $\alpha \in (0,1)$, we can choose $C_1>0$ such that 
$$
\vert I_1(s)\vert \le \frac{Ce^{MT}}{M} s^{\frac{\alpha}{2}} 
e^{s^{\frac{\alpha}2}-sT}\le Ce^{-C_1s}, \quad s > 2M.
$$

For $I_2$, since $u_0:=u(\cdot,0)\in C[0,1]$, we 
have
\begin{align*}
|I_2(s)| 
\le&\frac12 s^{\alpha-1} \|u_0\|_{C[0,1]}\int_0^1 e^{(\xi-1) 
s^{\frac{\alpha}2}}d\xi\\
=&\frac12 \|u_0\|_{C[0,1]}s^{\frac{\alpha}2-1}(1 - e^{-s^\frac{\alpha}2})
\le \frac12 \|u_0\|_{C[0,1]}s^{\frac{\alpha}2-1},\quad s>0.
\end{align*}
Finally, noting the equality \eqref{equ-sLg}, from \eqref{lim-sLg} and the 
estimates for $I_j$, $j=1,2$, we obtain an estimate for $I_3(s)$:
$$
|I_3(s)| \le Ce^{-C_1s}
+ \|g\|_{L^\infty(0,T)} s^{-1} 
+ \frac12 \|u_0\|_{C[0,1]}s^{\frac{\alpha}2 -1}, \quad s>2M,
$$
where $C_2:=\sup_{s\ge0}|I_1(s)|<\infty$, which further implies
$$
\left|\int_0^1 e^{(1-\xi) s^{\frac{\alpha}2}} u_0(\xi)d\xi \right|
\le C_2(s^{-\frac{\alpha}{2}} + s^{-\alpha}), \quad s>2M.
$$
The change of variables implies 
$$
\int^1_0 e^{\eta z}u_0(1-\eta) d\eta
= \int^1_0 e^{(1-\xi)z}u_0(\xi) d\xi.
$$
Therefore, after the change of variable $z:=s^{\frac{\alpha}2}$, we find
$$
\left|\int_0^1 e^{\eta z} u_0(1-\eta)d\eta \right|
\le C_2(z^{-1} + z^{-2}),\quad z>(2M)^{\frac{\alpha}2}.
$$
For $0 < z \le (2M)^{\frac{\alpha}2}$, we have
$$
\left\vert \int^1_0 e^{\eta z}u_0(1-\eta) d\eta\right\vert
\le \Vert u_0\Vert_{C[0,1]}e^z
\le \Vert u_0\Vert_{C[0,1]}e^{(2M)^{\frac{\alpha}2}}.
$$
Therefore we can choose constants $C_3 >0$ and $a \in (0,1)$ such that 
$$
\left\vert \int^1_0 e^{\eta z}u_0(1-\eta) d\eta\right\vert
\le C_3e^{az}, \quad z > 0.
$$
Hence Corollary \ref{coro} yields that $u_0\equiv0$ in $[0,1]$. From 
\eqref{equ-h}, it follows that 
$$
\int_0^t D_t^{1-\alpha}\theta_\alpha(1,\tau) g(t-\tau)d\tau=0,\quad t\in(0,T).
$$

Therefore, the Titchmarsh convolution theorem (see Doss \cite{D88} and 
Titchmarsh \cite{T26}) implies the existence of $T_1,T_2\ge0$ satisfying 
$T_1+T_2\ge T$ such that $D_t^{1-\alpha}\theta_\alpha(1,t)=0$ for almost all 
$t\in(0,T_1)$ and $g(t) = 0$ for all $t\in (0,T_2)$. However, recalling the 
definition of Wright function, we see that the Theta function 
$\theta_\alpha(t)$ is analytic in $t\in(0,T)$, hence $D_t^{1-\alpha}\theta_
\alpha(t)$ is $t$-analytic. 
Thus $T_1$ must be zero, that is, $g\equiv0$ in $(0,T)$. 
Finally we can prove the uniqueness of the solution of the initial-boundary 
value problem \eqref{equ-IBVP} similarly to \cite{SY11}.
Although in \cite{SY11}, the Dirichlet boundary condition is considered 
but the case of the Neumann boundary condition is treated in the same way.
Thus $u\equiv0$ in $[0,1]\times(0,T)$. This completes the proof of the 
lemma.
\end{proof}

\section{Concluding remarks}
\label{sec-rem}
In this paper, we first investigated the lateral Cauchy problem for the $1$-D 
time-fractional diffusion equation. On the basis of the Theta function method, 
we gave a representation formula of the solution and showed the uniqueness of 
the solution to the Cauchy problem by the use of the Laplace transform 
argument. As a direct conclusion of the uniqueness of the Cauchy problem, 
we proved that the classical unique continuation is valid.  Let us mention 
that the proof of the unique continuation principle heavily relies on the 
Theta function method which enable one to derive an explicit representation 
formula of the solution. It would 
be interesting to investigate what happens about the unique continuation 
property of the solution in the general dimensional case.

\section*{Acknowledgment}
The author thanks Grant-in-Aid for Research Activity Start-up 16H06712, JSPS.
This work is supported by Grant-in-Aid for Scientific Research (S) 15H05740, 
the A3 Foresight Program "Modeling and Computation of Applied Inverse 
Problems" by Japan Society for the Promotion of Science.


\end{document}